\documentclass[12pt]{article}
\usepackage{fullpage}
\usepackage{amssymb, amsmath, amsthm, enumerate, math tools}
\DeclareMathOperator{\Z}{\mathbb{Z}}

\DeclareMathOperator{\Q}{\mathbb{Q}}
\theoremstyle{plain}
\newtheorem{thm}{Theorem}[section]
\theoremstyle{definition}

\newtheorem{cor}[thm]{Corollary}
\newtheorem{lemma}[thm]{Lemma}
\newtheorem{conj}[thm]{Conjecture}
\begin{document}
\title{Irreducible quadratic polynomials and Euler's function}
\author{Noah Lebowitz-Lockard}
\maketitle
\begin{abstract} Let $P(x)$ be an irreducible quadratic polynomial in $\Z[x]$. We show that for almost all $n$, $P(n)$ does not lie in the range of Euler's totient function.
\end{abstract}

\section{Introduction} Let $V(x)$ be the number of $n \leq x$ in the range of Euler's $\varphi$-function. In 1929, Pillai proved that almost all numbers lie outside the range of the $\varphi$-function \cite{Pi}, namely that
\[V(x) = O\left(\frac{x}{(\log x)^{(\log 2)/e}}\right).\]
Multiple people (\cite{E}, \cite{EH1}, \cite{EH2}, \cite{P}, \cite{MP}) improved this bound. Ford established the order of magnitude of $V(x)$ \cite{F}:
\[V(x) = \Theta\left(\frac{x}{\log x} \exp(C(\log_3 x - \log_4 x)^2 + D\log_3 x - (D + (1/2) - 2C) \log_2 x)\right),\]
with $C \approx 0.82$ and $D \approx 2.18$.

For a given function $f$, we define
\[V_f (x) = \#\{n \leq x: \exists m\ \textrm{s.t.}\ \varphi(m) = f(n)\}.\]
Pollack and Pomerance proved that almost all squares lie outside the range of the $\varphi$-function \cite{PolP}. Specifically, for $f(x) = x^2$,
\[\frac{x}{(\log x)^2 (\log \log x)^2} \ll V_f (x) \ll \frac{x}{(\log x)^{0.0063}}.\]
Let $P(x) = ax^2 + bx + c$ be an irreducible quadratic polynomial. We show that
\[V_P (x) = O\left(\frac{x}{(\log x)^{0.03}}\right).\]
Hence, for almost all $n$, $P(n)$ lies outside the range of the totient function.

The only odd number in the range of the totient function is $1$. If $P(x)$ only takes odd values, then $V_P (x)$ is the number of positive solutions $n \leq x$ of $P(n) = 1$. In this case, $V_P (x) \leq 2$. We also show that if $P(x)$ is never a multiple of $4$, then $V_P (x) \ll x/\log x$. Finally, we improve our bounds on $V_P (x)$ assuming the abc Conjecture.

\section{Outline}

Suppose $P(n)$ in the range of the $\varphi$-function. Let $p$ be the largest prime number for which there exists a number $m$ such that $p | m$ and $\varphi(m) = P(n)$. By definition, $p - 1 | P(n)$. We write $P(n) = (p - 1)v$. We choose a number $T = o(x)$, which we will optimize later. There are three cases:
\begin{enumerate}
\item{$p > 4ax,$}
\item{$T < p \leq 4ax,$}
\item{$p \leq T.$}
\end{enumerate}

For a given number $k$, let $\rho(k)$ be the number of solutions to the congruence $P(n) \equiv 0\ \textrm{mod}\ k$. Note that $\rho$ is a multiplicative function. Let $D$ be the discriminant of $P(x)$. If a prime $q$ does not divide $2a$, then the solutions to $P(x) \equiv 0\ \textrm{mod}\ q$ are
\[x \equiv \frac{-b \pm \sqrt{D}}{2a}\ \textrm{mod}\ q.\]
Hence, for a given $q \nmid 2aD$,
\[\rho(q) = \left\{\begin{array}{ll}
2, & \textrm{if}\ \left(\frac{D}{q}\right) = 1, \\
0, & \textrm{if}\ \left(\frac{D}{q}\right) = -1.
\end{array}\right.\]

For all but finitely many $q$, $q \nmid 2aD$. By the Chebotarev Density Theorem, the primes which split in $\Q[\sqrt{D}]$ and the primes which are inert in $\Q[\sqrt{D}]$ both have density $1/2$. In other words,
\[\lim_{x \to \infty} \frac{1}{\pi(x)} \#\left\{q \leq x: \left(\frac{D}{q}\right) = 1\right\} = \lim_{x \to \infty} \frac{1}{\pi(x)} \#\left\{q \leq x: \left(\frac{D}{q}\right) = -1\right\} = \frac{1}{2}.\]

\section{A large factor of the form $p - 1$}

Let $V_1$ be the number of $n \leq x$ for which $p > 4ax$.

\begin{thm} We have
\[V_1 = O\left(\frac{x(\log \log x)^5}{(\log x)^{1 - (e(\log 2)/2)}}\right).\]
\end{thm}

\begin{proof} We write $\varphi(m) = P(n)$ with $p | m$ for some $p > 4ax$. We first bound $m$. Note that $P(n) = an^2 + bn + c \leq 2an^2 \leq 2ax^2$ for $x$ sufficiently large. By \cite[Theorem $328$]{HarWr},
\[\liminf_{k \to \infty} \frac{\varphi(k) \log \log k}{k} = e^{-\gamma},\]
where $\gamma$ is the Euler-Mascheroni constant. Thus, $m \ll x^2 \log \log x$.

By partial summation, the number of $m \ll x^2 \log \log x$ with a divisor of the form $p^2$ with $p > 4ax$ is $O(x \log \log x/\log x)$. Hence, we may assume that $p^2$ does not divide $m$. We write $m = pr$ with $p \nmid r$. So, $\varphi(m) = P(n) = (p - 1)v$ with $\varphi(r) = v$. Because $p > 4ax$ and $P(n) \leq 2ax^2$, $v < x/2$ as well.

We write
\[n \equiv t_1, \ldots, t_{\rho(v)}\ \textrm{mod}\ v,\]
with $0 \leq t_i < v$ for all $i \leq \rho(v)$. Fix $i$ and let $t = t_i$. Let $n = uv + t$. We have
\begin{eqnarray*}
p & = & \frac{P(n)}{v} + 1 \\
& = & \frac{P(uv + t)}{v} + 1 \\
& = & \frac{a(uv + t)^2 + b(uv + t) + c}{v} + 1 \\
& = & avu^2 + (2at + b)u + \left(\frac{at^2 + bt + c}{v} + 1\right).
\end{eqnarray*}
So, we can recast the problem in terms of $u$. Given $v$ and $a$, we look for the number of values of $u$ for which the quadratic expression above is prime, then sum over all $v$ and $a$. In other words, we want to bound the size of
\[M = M_{v, t} = \{u \leq x/v : R(u)\ \textrm{is prime}\},\]
where
\[R(u) = avu^2 + (2at + b)u + \left(\frac{at^2 + bt + c}{v} + 1\right).\]

The discriminant of $R$ is $D - 4av$. If $R$ is reducible, then $D - 4av$ is a square. The number of $v$ for which $D - 4av$ is non-negative is $O(1)$ for $P$ fixed. For each value of $v$, the number of corresponding $n$ is also $O(1)$ with respect to $P$. Because there are $O(1)$ values of $n$ for which $R$ is reducible, we assume that $R$ is irreducible. Brun's Sieve \cite[Theorem $2.6$]{HalbR} gives us
\[\# M \ll \frac{x}{v} \prod_{\substack{q < x/v \\ \rho_R (q) \neq q}} \left(1 - \frac{\rho_R (q)}{q}\right),\]
where $\rho_R (q)$ the number of solutions to $R(u) \equiv 0\ \textrm{mod}\ q$ for a given prime $q$.

The number of possible $n$ is the sum of $\# M$ over all possible $v$ and $t$. In addition, $v$ lies in the range of Euler's function. For notational convenience, we let $\sum'$ have the condition that $D - 4av$ is not a square. We have
\[V_1 \ll \sideset{}{'}\sum_{\substack{v < x/2 \\ v \in \varphi(\Z_+)}} \sum_{\substack{0 \leq t < v \\ P(t) \equiv 0 (v)}} \frac{x}{v} \prod_{\substack{2 < q < x/v \\ \rho_R (q) \neq q}} \left(1 - \frac{\rho_R (q)}{q}\right).\]

We now bound
\[\sideset{}{'}\sum_{\substack{v < x/2 \\ v \in \varphi(\Z_+)}} \sum_{\substack{0 \leq t < v \\ P(t) \equiv 0(v)}} \frac{x}{v} \prod_{\substack{q < x/v \\ \rho_R (q) \neq q}} \left(1 - \frac{\rho_R (q)}{q}\right) \ll x\sideset{}{'}\sum_{\substack{v < x/2 \\ v \in \varphi(\Z_+)}} \frac{\rho(v)}{v} \prod_{\substack{2 < q < x/v \\ q \nmid av(D - 4av) \\ \left(\frac{D - 4av}{q}\right) = 1}} \left(1 - \frac{2}{q}\right).\]
For the product, we multiply by a similar product over the $q$ dividing $2av(D - 4av)$ in order to make it easier to manipulate:
\[x\sideset{}{'}\sum_{\substack{v < x/2 \\ v \in \varphi(\Z_+)}} \frac{\rho(v)}{v} \prod_{\substack{q < x/v \\ \left(\frac{D - 4av}{q}\right) = 1}} \left(1 - \frac{2}{q}\right) \prod_{\substack{2 < q < x/v \\ q | av(D - 4av) \\ \left(\frac{D - 4av}{q}\right) = 1}} \left(1 - \frac{2}{q}\right)^{-1}.\]
We simplify the second product as follows:
\begin{eqnarray*}
\prod_{\substack{q < x/v \\ q | av(D - 4av) \\ \left(\frac{D - 4av}{q}\right) = 1}} \left(1 - \frac{2}{q}\right)^{-1} & \ll & \prod_{q | v(D - 4av)} \left(1 - \frac{1}{q}\right)^{-2} \\
& = & \left(\frac{v(D - 4av)}{\varphi(v|D - 4av|)}\right)^2 \\
& \ll & (\log \log (v|D - 4av|))^2 \\
& \ll & (\log \log v)^2.
\end{eqnarray*}
We now have
\[V_1 \ll x \sideset{}{'}\sum_{\substack{v < x/2 \\ v \in \varphi(\Z_+)}} \frac{\rho(v) (\log \log v)^2}{v} \prod_{\substack{q < x/v \\ \left(\frac{D - 4av}{q}\right) = 1}} \left(1 - \frac{2}{q}\right).\]

For small $v$ it is not difficult to show that $D - 4av$ is a quadratic residue mod $q$ for about half of all $q < x/v$. Unfortunately, $v$ may be large enough relative to $x$ that this is not always true. We bound the product from above:
\begin{eqnarray*}
\prod_{\substack{q < x/v \\ \left(\frac{D - 4av}{q}\right) = 1}} \left(1 - \frac{2}{q}\right) & = & \prod_{2 < q < x/v} \left(1 - \frac{1}{q} \left(1 + \left(\frac{D - 4av}{q}\right)\right)\right) \prod_{\substack{2 < q < x/v \\ q | D - 4av}} \left(1 - \frac{1}{q}\right)^{-1} \\
& \ll & \frac{|D - 4av|}{\varphi(|D - 4av|)} \prod_{2 < q < x/v} \left(1 - \frac{1}{q}\right) \prod_{2 < q < x/v} \left(1 - \frac{1}{q} \left(\frac{D - 4av}{q}\right)\right) \\
& \ll & \frac{\log \log v}{\log (x/v)} \prod_{2 < q < x/v} \left(1 - \frac{1}{q} \left(\frac{D - 4av}{q}\right)\right).
\end{eqnarray*}
Therefore,
\begin{eqnarray*}
V_1 & \ll & x\sideset{}{'}\sum_{\substack{v < x/2 \\ v \in \varphi(\Z_+)}} \frac{\rho(v) (\log \log v)^3}{v\log(x/v)} \prod_{2 < q < x/v} \left(1 - \frac{1}{q} \left(\frac{D - 4av}{q}\right)\right) \\
& \ll & x(\log \log x)^3 \sideset{}{'}\sum_{\substack{v < x/2 \\ v \in \varphi(\Z_+)}} \frac{\rho(v)}{v\log(x/v)} \prod_{2 < q < x/v} \left(1 - \frac{1}{q} \left(\frac{D - 4av}{q}\right)\right).
\end{eqnarray*}

We combine Lemmas $6$ and $8$ of \cite{PolP} into one result and apply this result to the Kronecker symbol.

\begin{lemma} For all squarefree $d$ and $\epsilon > 0$,
\[\prod_{2 < q \leq y} \left(1 - \frac{1}{q} \left(\frac{d}{q}\right)\right) = O(d^\epsilon).\]
In addition, the number of (not necessarily squarefree) $d \leq x$ for which 
\[\prod_{2 < q \leq y} \left(1 - \frac{1}{q} \left(\frac{d}{q}\right)\right) \leq (\log \log |3d|)^2\]
does not hold for some $y$ is $O(x^\epsilon)$.
\end{lemma}

If $q \nmid D - 4av$ and $d$ is the squarefree part of $D - 4av$, then
\[\left(\frac{D - 4av}{q}\right) = \left(\frac{d}{q}\right).\]
When $d$ is the squarefree part of $D - 4av$,
\begin{eqnarray*}
\prod_{2 < q \leq y} \left(1 - \frac{1}{q} \left(\frac{D - 4av}{q}\right)\right) & = & \prod_{\substack{2 < q \leq y \\ q \nmid D - 4av}} \left(1 - \frac{1}{q} \left(\frac{d}{q}\right)\right) \\
& = & \prod_{\substack{2 < q \leq y \\ q | D - 4av}} \left(1 - \frac{1}{q} \left(\frac{d}{q}\right)\right)^{-1} \prod_{2 < q \leq y} \left(1 - \frac{1}{q} \left(\frac{d}{q}\right)\right) \\
& \leq & \prod_{\substack{2 < q \leq y \\ q | D - 4av}} \left(1 - \frac{1}{q}\right)^{-1} \prod_{2 < q \leq y} \left(1 - \frac{1}{q} \left(\frac{d}{q}\right)\right) \\
& = & \frac{D - 4av}{\varphi(|D - 4av|)} \prod_{2 < q \leq y} \left(1 - \frac{1}{q} \left(\frac{d}{q}\right)\right) \\
& \ll & (\log \log |3(D - 4av)|) \prod_{2 < q \leq y} \left(1 - \frac{1}{q} \left(\frac{d}{q}\right)\right).
\end{eqnarray*}

For a given squarefree number $d$, the number of numbers $\leq x$ with squarefree part $d$ is $O(x^{1/2})$.  For all but $O(x^{(1/2) + \epsilon})$ numbers $v \leq x/2$,
\[\prod_{2 < q \leq y} \left(1 - \frac{1}{q} \left(\frac{D - 4av}{q}\right)\right) \leq (\log \log |3(D - 4av)|)^3.\]

Let $S(k)$ be the squarefree part of $k$. We split our sum into two parts.

Suppose $S(D - 4av) \notin \mathcal{D}$:
\[\sideset{}{'}\sum_{\substack{v < x/2 \\ v \in \varphi(\Z_+)}} \frac{\rho(v)}{v \log(x/v)} \prod_{2 < q < x/v} \left(1 - \frac{1}{q} \left(\frac{D - 4av}{q}\right)\right) \ll (\log \log x)^3 \sum_{\substack{v < x/2 \\ v \in \varphi(\Z_+)}} \frac{\rho(v)}{v\log(x/v)}.\]
We bound this sum using dyadic intervals:
\begin{eqnarray*}
\sum_{\substack{v < x/2 \\ v \in \varphi(\Z_+)}} \frac{\rho(v)}{v\log(x/v)} & = & \sum_{i < \log x/\log 2} \sum_{\substack{2^i < x/v \leq 2^{i + 1} \\ v \in \varphi(\Z_+)}} \frac{\rho(v)}{v\log(x/v)} \\
& \ll & \sum_{i < \log x/\log 2} \frac{2^i}{x\log(2^i)} \sum_{\substack{2^i < x/v \leq 2^{i + 1} \\ v \in \varphi(\Z_+)}} \rho(v) \\
& \ll & \sum_{i < \log x/\log 2} \frac{1}{i} \left(\frac{1}{x/2^i}\right) \sum_{\substack{v < x/2^i \\ v \in \varphi(\Z_+)}} \rho(v).
\end{eqnarray*}

We bound the sum of the $\rho(v)$ terms using H{\" o}lder's Inequality. Let $A, B > 1$ satisfy $(1/A) + (1/B) = 1$. Recall that $V(x)$ is the number of $n \leq x$ in the range of $\varphi$. For the following equation, we use the fact that $V(x) \ll x/(\log x)^{1 - \epsilon}$ for all $\epsilon > 0$. We have
\begin{eqnarray*}
\sum_{\substack{v < x/2^i \\ v \in \varphi(\Z_+)}} \rho(v) & \ll & \left(\sum_{v < x/2^i} \rho(v)^A\right)^{1/A} \left(\sum_{\substack{v < x/2^i \\ v \in \varphi(\Z_+)}} 1^B\right)^{1/B} \\
& \ll & \left(\sum_{v < x/2^i} \rho(v)^A\right)^{1/A} (V(x/2^i))^{1/B} \\
& \ll & \left(\sum_{v < x/2^i} \rho(v)^A\right)^{1/A} \left(\frac{x/2^i}{(\log(x/2^i))^{1 - \epsilon}}\right)^{1/B}.
\end{eqnarray*}
In order to bound the sum of $\rho(v)^A$, we use the following Brun-Titchmarsh-like theorem for multiplicative functions (the $k = 1$, $y = x$ cases of \cite{S}, \cite{Pol}).
\begin{thm} Let $f$ be a non-negative multiplicative function satisfying the following conditions:
\begin{enumerate}
\item{There is a positive constant $A_1$ such that $f(p^r) \leq A_1^r$ for all prime $p$ and non-negative $r$.}
\item{For all $\epsilon > 0$, there is a positive constant $A_2 = A_2 (\epsilon)$ for which $f(n) \leq A_2 n^\epsilon$ for all $n$.}
\end{enumerate}
\begin{enumerate}[(a)]
\item{We have
\[\sum_{n \leq x} f(n) \ll \frac{x}{\log x} \exp\left(\sum_{p \leq x} \frac{f(p)}{p}\right).\]}
\item{In addition,
\[\sum_{p \leq x} f(p - 1) \ll \frac{x}{\log x} \exp\left(\sum_{p \leq x} \frac{f(p) - 1}{p}\right).\]}
\end{enumerate}
\end{thm}

We show that $\rho$ satisfies the conditions of this theorem. For a given prime $p$, let $p^{\sigma_1} \| D$. It is well-known (see \cite[Theorem $53$]{N}) that if some coefficient of $P(x)$ is not a multiple of $p$, then $\rho(p^r) \leq \rho(p^{2\sigma_1 + 1})$. Suppose $P(x) = p^{\sigma_2} Q(x)$ where $\sigma_2$ is maximal, i.e. some coefficient of $Q(x)$ is not a multiple of $p$. For all $r \geq \sigma_2$, $\rho(p^r) = \rho_Q (p^{r - \sigma_2})$ because each solution to the congruence $Q(x) \equiv 0\ \textrm{mod}\ p^{r - \sigma_2}$ lifts to a solution of $P(x) \equiv 0\ \textrm{mod}\ p^r$. (If $r \leq \sigma_2$, then $\rho(p^r) = p^r \leq p^{\sigma_2})$. So,
\[\rho(p^r) = \rho_Q (p^{r - \sigma_2}) \leq \rho_Q (p^{2(\sigma_1 - 2\sigma_2) + 1})\]
because the discriminant of $Q(x)$ is $D/p^{2\sigma_2}$. For all $r$,
\[\rho(p^r) \leq \max(\rho_Q (p^{2(\sigma_1 - 2\sigma_2 + 1}), p^{\sigma_2}).\]
For all but finitely many $p$, $\sigma_1 \leq 2$. Thus, $\rho(p^r)$ is bounded by a constant $C$, giving us $(1)$. Let $\omega(n)$ be the number of distinct prime factors of $n$. We have
\[\rho(n) \leq C^{\omega(n)} \ll C^{\log n/\log \log n} = o(n^\epsilon)\]
for all $\epsilon > 0$, implying $(2)$.

Therefore,
\begin{eqnarray*}
\sum_{v < x/2^i} \rho(v)^A & \ll & \frac{x/2^i}{\log(x/2^i)} \exp\left(\sum_{q < x/2^i} \frac{\rho(q)^A}{q}\right) \\
& \ll & \frac{x/2^i}{\log(x/2^i)} \exp\left(\sum_{q | 2aD} \frac{q^A}{q} + \sum_{\substack{q < x/2^i \\ q \nmid 2aD \\ \left(\frac{D}{q}\right) = 1}} \frac{2^A}{q}\right) \\
& \ll & \frac{x/2^i}{\log(x/2^i)} \exp\left(\sum_{\substack{q < x/2^i \\ \left(\frac{D}{q}\right) = 1}} \frac{2^A}{q}\right) \\
& \ll & \frac{x/2^i}{\log(x/2^i)} \exp(2^{A - 1} \log \log(x/2^i)) \\
& \ll & (x/2^i)(\log(x/2^i))^{2^{A - 1} - 1}.
\end{eqnarray*}
Plugging this into our earlier inequality gives us
\[\sum_{\substack{v < x/2^i \\ v \in \varphi(\Z_+)}} \rho(v) \ll \left(\frac{x}{2^i}\right) \left(\log\left(\frac{x}{2^i}\right)\right)^{\frac{2^{A - 1} - 1}{A} - \frac{1}{B} + \frac{\epsilon}{B}} = \left(\frac{x}{2^i}\right)\left(\log\left(\frac{x}{2^i}\right)\right)^{\frac{2^A}{2A} - 1 + \left(1 - \frac{1}{A}\right)\epsilon}.\]
The minimum value of $(2^A/(2A)) - 1$ is $((e\log 2)/2) - 1 < 0$, which occurs at $A = 1/\log 2$. Hence,
\[\sum_{\substack{v < x/2^i \\ v \in \varphi(\Z_+)}} \rho(v) \ll \frac{x/2^i}{(\log(x/2^i))^{1 - ((e\log 2)/2) - (1 - \log 2)\epsilon}},\]
giving us
\begin{eqnarray*}
\sum_{\substack{v < x/2 \\ v \in \varphi(\Z_+)}} \frac{\rho(v)}{v\log(x/v)} & \ll & \sum_{i < \log x/\log 2} \frac{1}{i} \left(\frac{1}{x/2^i}\right) \sum_{\substack{v < x/2^i \\ v \in \varphi(\Z_+)}} \rho(v) \\
& \ll & \sum_{i < \log x/\log 2} \frac{1}{i (\log(x/2^i))^{1 - ((e\log 2)/2) - (1 - \log 2)\epsilon}}.
\end{eqnarray*}

For notational convenience, we replace $\epsilon$ with $(1 - \log 2)\epsilon$. We may now finish off our dyadic interval. In order to bound this sum, we split it into two cases: $i > K$ and $i < K$, with $K = (\log x)^{O(1)}$:
\[\sum_{i < K} \frac{1}{i(\log(x/2^i))^{1 - (e(\log 2)/2) - \epsilon}} \ll \sum_{i < K} \frac{1}{i(\log(x/2^K))^{1 - (e(\log 2)/2) - \epsilon}} \ll \frac{\log K}{(\log(x/2^K))^{1 - (e(\log 2)/2) - \epsilon}},\]
\[\sum_{K < i < \log x/\log 2} \frac{1}{i(\log(x/2^i))^{1 - (e(\log 2)/2) - \epsilon}} \ll \sum_{i < \log x/\log 2} \frac{1}{K} \ll \frac{\log x}{K}.\]
Setting the two sums equal to each other suggests choosing $K = (\log x)^{e(\log 2)/2}$. This yields
\[\sum_{\substack{v < x/2 \\ S(D - 4av) \notin \mathcal{D}}} \sum_t \# M_{v, t} \ll \frac{x}{(\log x)^{1 - e(\log 2)/2 - \epsilon}}.\]

Suppose $S(D - 4av) \in \mathcal{D}$. Let $U$ be a function of $x$ chosen with $U = O(x^\epsilon)$ for all $\epsilon$. Suppose $v \leq U$. We want to bound
\[(\log \log x)^3 \sideset{}{'}\sum_{v \leq U} \frac{\rho(v)}{v\log(x/v)} \prod_{q < x/v} \left(1 - \frac{1}{q} \left(\frac{D - 4av}{q}\right)\right).\]
By Lemma $3.2$, the product above is $O(v^\epsilon)$ for any $\epsilon > 0$. In addition, $\log(x/v) \gg \log x$ because $v \leq U$. We already established that $\rho(v) \ll v^\epsilon$. Putting this together, we have
\[\sideset{}{'}\sum_{v \leq U} \frac{\rho(v)}{v\log(x/v)} \prod_{q < x/v} \left(1 - \frac{1}{q} \left(\frac{D - 4av}{q}\right)\right) \ll \sum_{v < U} \frac{1}{v^{1 - 2\epsilon} \log x} \ll \frac{U^{2\epsilon}}{\log x}.\]

Now, we consider the case where $S(D - 4av) \in \mathcal{D}$ and $U < v < x/2$. We have
\[\sideset{}{'}\sum_{\substack{U < v < x/2 \\ S(D - 4av) \in \mathcal{D}}} \frac{\rho(v)}{v\log(x/v)} \prod_{q < x/v} \left(1 - \frac{1}{q} \left(\frac{D - 4av}{q}\right)\right) \ll \sum_{\substack{U < v < x/2 \\ S(D - 4av) \in \mathcal{D}}} \frac{1}{v^{1 - 2\epsilon} \log(x/v)}.\]
Because $v < x/2$, $\log(x/v) \gg 1$. At this point, we use dyadic intervals:
\begin{eqnarray*}
\sum_{\substack{U < v < x/2 \\ S(D - 4av) \in \mathcal{D}}} \frac{1}{v^{1 - 2\epsilon}} & \ll & \sum_i \sum_{\substack{2^i U < v < 2^{i + 1} U \\ S(D - 4av) \in \mathcal{D}}} \frac{1}{v^{1 - 2\epsilon}} \\
& \ll & \frac{1}{U^{1 - 2\epsilon}} \sum_i \sum_{\substack{v < 2^{i + 1} U \\ S(D - 4av) \in \mathcal{D}}} \frac{1}{2^{(1 - 2\epsilon) i}} \\
& \ll & \frac{1}{U^{1 - 2\epsilon}} \sum_i \frac{(2^{i + 1} U)^{(1/2) + 2\epsilon}}{2^{(1 - 2\epsilon)i}} \\
& \ll & \frac{1}{U^{(1/2) - 4\epsilon}} \sum_i \frac{1}{2^{((1/2) - 2\epsilon)i}} \\
& \ll & \frac{1}{U^{(1/2) - 4\epsilon}}.
\end{eqnarray*}

We add our sums for $v < U$ and $v \geq U$ together:
\[\sideset{}{'}\sum_{\substack{v < x/2 \\ S(D - 4av) \in \mathcal{D}}} \frac{\rho(v)}{v \log(x/v)} \prod_{q < x/v} \left(1 - \frac{1}{q} \left(\frac{D - 4av}{q}\right)\right) \ll \frac{U^{2\epsilon}}{\log x} + \frac{1}{U^{(1/2) - 4\epsilon}}.\]

We choose $U$ so that
\[\frac{1}{\log x} = \frac{1}{U^{1/2}}.\]
Thus,
\[U = (\log x)^2\]
and
\begin{eqnarray*}
\sideset{}{'}\sum_{\substack{v < x/2 \\ S(D - 4av) \in \mathcal{D}}} \frac{\rho(v) (\log \log v)^3}{v \log(x/v)} \prod_{q < x/v} \left(1 - \frac{1}{q} \left(\frac{D - 4av}{q}\right)\right) & \ll & \frac{1}{(\log x)^{1 - 4\epsilon}} + \frac{1}{(\log x)^{1 - 8\epsilon}} \\
& \sim & \frac{1}{(\log x)^{1 - 8\epsilon}}.
\end{eqnarray*}

We have obtained the following bound:
\[V_1 = O\left(\frac{x}{(\log x)^{1 - e(\log 2)/2 - \epsilon}} + \frac{x(\log \log x)^3}{(\log x)^{1 - 8\epsilon}}\right) = O\left(\frac{x}{(\log x)^{1 - e(\log 2)/2 - \epsilon}}\right).\]
\end{proof}

\section{A factor of the form $p - 1$ in the interval $(T, 4ax)$ I}

In the next two sections, we assume that $T < p \leq 4ax$. In addition, fix a number $A \in (1/2, 1)$. Let $\Omega_T (y)$ be the number of (not necessarily distinct) prime factors of $y$ that are smaller than $T$. We define $V_2$ as the number of $n \leq x$ for which $T < p < 4ax$ and $\Omega_T (p - 1) < A \log \log T$.

\begin{thm} For all $A \in (1/2, 1)$, we have
\[V_2 = O\left(\frac{x}{(\log T)^{A\log A - A + 1}}\right).\]
\end{thm}

\begin{proof} Given $p$, we can bound the number of $n \leq x$ for which $p - 1$ divides $P(n)$. The number of $n \leq x$ for which $p - 1 | P(n)$ is
\[\frac{x\rho(p - 1)}{p - 1} + O(\rho(p - 1)).\]
In order to bound the number of possible $n$ for any given $p$ satisfying the conditions above, we sum over all possible $p$. We obtain
\[V_2 \leq \sum_{\substack{T < p < 4ax \\ \Omega_T (p - 1) < A \log \log T}} \left(\frac{x\rho(p - 1)}{p - 1} + O(\rho(p - 1))\right).\]

We have $\rho(p - 1) < (1/(4a)) x\rho(p - 1)/(p - 1)$. So, we only need to consider the first term of the sum in order to bound the order of magnitude:
\[V_2 \ll x\sum_{\substack{T < p < 4ax \\ \Omega_T (p - 1) < A \log \log T}} \frac{\rho(p - 1)}{p - 1}.\]

Fix a constant $B < 1$. Because $\Omega_T (p - 1) < A \log \log T$,
\[B^{\Omega_T (p - 1)} > B^{A \log \log T} = (\log T)^{A \log B}.\]
For each prime $p$ in our sum,
\[\frac{B^{\Omega_T (p - 1)}}{(\log T)^{A \log B}} > 1.\]
Multiplying every term in our sum by this quantity will increase the sum. Hence,
\begin{eqnarray*}
\sum_{\substack{T < p < 4ax \\ \Omega_T (p - 1) < A \log \log T}} \frac{\rho(p - 1)}{p - 1}	& \leq & \sum_{\substack{T < p < 4ax \\ \Omega_T (p - 1) < A \log \log T}} \frac{\rho(p - 1)}{p - 1} \left(\frac{B^{\Omega_T (p - 1)}}{(\log T)^{A \log B}}\right) \\
& \leq & \frac{1}{(\log T)^{A \log B}} \sum_{T < p < 4ax} \frac{B^{\Omega_T (p - 1)} \rho(p - 1)}{p - 1}.
\end{eqnarray*}
Let $k = \log 2$. In order to evaluate this sum, we break it into dyadic intervals:
\begin{eqnarray*}
\sum_{T < p < 4ax} \frac{B^{\Omega_T (p - 1)} \rho(p - 1)}{p - 1} & \leq & \sum_{0 \leq i < k\log(4ax/T) + 1} \sum_{2^i T \leq p < 2^{i + 1} T} \frac{B^{\Omega_T (p - 1)} \rho(p - 1)}{p - 1} \\
& \ll & \sum_{0 \leq i < k\log(4ax/T) + 1} \frac{1}{2^i T} \sum_{2^i T \leq p < 2^{i + 1} T} B^{\Omega_T (p - 1)} \rho(p - 1).
\end{eqnarray*}
By Theorem $3.3$,
\begin{eqnarray*}
\sum_{2^i T \leq p < 2^{i + 1} T} B^{\Omega_T (p - 1)} \rho(p - 1) & \ll & \frac{2^{i + 1} T}{\log(2^{i + 1} T)} \exp\left(\sum_{p < 2^{i + 1} T} \frac{B^{\Omega_T (p)} \rho(p)}{p} - \sum_{p < 2^i T} \frac{1}{p}\right) \\
& \ll & \frac{2^i T}{\log(2^i T)} \exp\left(\sum_{p | 2aD} \frac{B}{p} + \sum_{\substack{p \leq T \\ \left(\frac{D}{p}\right) = 1}} \frac{2B}{p} + \sum_{\substack{T < p < 2^{i + 1} T \\ \left(\frac{D}{p}\right) = 1}} \frac{2}{p} - \sum_{p < 2^i T} \frac{1}{p}\right) \\
& \ll & \frac{2^i T}{\log(2^i)} \exp(\log \log (2^{i + 1} T) - \log \log (2^i T) - (1 - B) \log \log T) \\
& \ll & \frac{2^i T}{i} \exp(-(1 - B) \log \log T) \\
& \sim & \frac{2^i T}{i(\log T)^{1 - B}}. \\
\end{eqnarray*}
Hence,
\[\sum_{T < p < 4ax} \frac{B^{\Omega_T (p - 1)} \rho(p - 1)}{p - 1} \ll \sum_{i < k\log(4ax/T) + 1} \frac{1}{i(\log T)^{1 - B}} \ll \frac{\log \log x}{(\log T)^{1 - B}}.\]

Putting all this together shows us that
\[V_2 \ll \frac{x}{(\log T)^{A\log B - B + 1}}.\]
We fix $A$ and let $B = A$ to make $A\log B - B + 1$ as large as possible. Hence,
\[V_2 = O\left(\frac{x}{(\log T)^{A\log A - A + 1}}\right).\]
\end{proof}

Note that $A\log A - A + 1$ is positive for all $A \in (1/2, 1)$.

\section{A factor of the form $p - 1$ in the interval $(T, 4ax)$ II}

Let $V_3$ be the number of $n \leq x$ for which $T < p \leq 4ax$ and $\Omega_T (p - 1) > A \log \log T$.

\begin{thm} We have
\[V_3 = O\left(\frac{x}{(\log T)^{(A + (1/2))\log(A + (1/2)) - A + (1/2)}}\right).\]
\end{thm}

To prove this theorem, we must show two preliminary results. Suppose $P(n) = (p - 1)(q - 1)v$ with $p, q > T$ and $\Omega_T (p - 1), \Omega_T (q - 1) > A \log \log T$. Then, $\Omega_T (P(n)) > 2A \log \log T$. We bound the number of such $n$ with the following results.
\begin{lemma} For all $\epsilon > 0$, the number of $n \leq x$ for which $\omega_T (P(n)) > (1 + \epsilon) \log \log T$ is
\[O\left(\frac{x}{(\log T)^{(1 + \epsilon) \log(1 + \epsilon) - \epsilon}}\right).\]
\end{lemma}

\begin{proof} Fix $z > 1$. We bound the sum of $z^{\omega_T (P(n))}$. Note that $f(n) = z^{\omega_T (n)}$ is a non-negative multiplicative function. In addition, $f(p^\ell)$ is $1$ or $z$ for all $p$ and $\ell$. We can also show that $f(n) \ll n^\epsilon$ for all $\epsilon > 0$: 
\[f(n) = z^{\omega_T (n)} \ll z^{\omega (n)} \ll z^{\log n/\log \log n} = n^{\log z/\log \log n} \ll n^\epsilon.\]

By a result of Nair \cite{Nai},
\[\sum_{n \leq x} z^{\omega_T (P(n))} \ll x \prod_{q \leq x} \left(1 - \frac{\rho(q)}{q}\right)\exp\left(\sum_{q \leq x} \frac{z^{\omega_T (q)} \rho(q)}{q}\right).\]
We have
\begin{eqnarray*}
\sum_{n \leq x} z^{\omega_T (P(n))} & \ll & x \prod_{\substack{q \leq x \\ \left(\frac{D}{q}\right) = 1}} \left(1 - \frac{2}{q}\right) \exp\left(\sum_{q | 2aD} \frac{z}{q} + \sum_{\substack{q < T \\ \left(\frac{D}{q}\right) = 1}} \frac{2z}{q} + \sum_{\substack{T < q \leq x \\ \left(\frac{D}{q}\right) = 1}} \frac{2}{q}\right) \\
& \ll & x\left(\frac{1}{\log x}\right)\exp(\log \log x + (z - 1)\log \log T)) \\
& = & x(\log T)^{z - 1}.
\end{eqnarray*}

Let $M$ be the number of $n \leq x$ for which $\omega_T (P(n)) > (1 + \epsilon) \log \log T$. Then,
\[\sum_{n \leq x} z^{\omega_T (P(n))} \geq z^{(1 + \epsilon) \log \log T} M = (\log T)^{(1 + \epsilon) \log z} M.\]
Combining our two bounds gives us
\[M \ll x(\log T)^{z - (1 + \epsilon)(\log z) - 1}\]
We can choose $z$ to minimize the exponent. At the minimum, $z = 1 + \epsilon$, giving us
\[M \ll \frac{x}{(\log T)^{(1 + \epsilon)\log(1 + \epsilon) - \epsilon}}.\]
\end{proof}

\begin{thm} For all $C, \delta > 0$, the number of $n \leq x$ for which $P(n)$ has a square divisor greater than $(\log T)^C$ is
\[O\left(\frac{x}{(\log T)^{(1 - \delta)C/2}}\right).\]
\end{thm}

\begin{proof} Suppose $r^2 | P(n)$ with $r^2 > (\log T)^C$. Assume $r^2 \leq x^{2 - \epsilon}$ for a fixed $\epsilon > 0$. The number of possible $n \leq x$ is
\[\sum_{r : (\log T)^C < r^2 \leq x^{2 - \epsilon}} \left(\frac{x\rho(r^2)}{r^2} + O(\rho(r^2))\right).\]

For all $\epsilon > 0$, $\rho(r^2) \ll r^\delta$. Therefore,
\[\sum_{(\log T)^C < r^2 \leq x^{2 - \epsilon}} \frac{x\rho(r^2)}{r^2} \ll \sum_{r > (\log T)^{C/2}} \frac{x}{r^{2 - \delta}} \sim \frac{x}{(\log T)^{(1 - \delta)C/2}}\]
and
\[\sum_{(\log T)^C < r^2 \leq x^{2 - \epsilon}} \rho(r^2) \ll \sum_{r \leq x^{1 - (\epsilon/2)}} r^\delta \ll x^{1 + \delta - (\epsilon/2)}.\]
If $\epsilon > 2\delta$, then the second sum is smaller than a constant multiple of the first one.

We may assume that $r^2 > x^{2 - \epsilon}$. If $r$ has a divisor $d \in ((\log T)^{C/2}, x^{1 - (\epsilon/2)}]$, then $P(n)$ has a square divisor in the range $((\log T)^C, x^{2 - \epsilon}]$, which we have already discussed. Suppose otherwise. Let $p$ be a prime factor of $r$. If $p \in (x^{\epsilon/2}, x^{1 - (\epsilon/2)}/(\log T)^{C/2}]$, then $r/p \in ((\log T)^{C/2}, x^{1 - (\epsilon/2)}]$. We may assume that if $p | r$, then $p \leq x^{\epsilon/2}$ or $p > x^{1 - (\epsilon/2)}/(\log T)^{C/2}$. If every prime factor is $\leq x^{\epsilon/2}$, then $r$ has a divisor in the range $((\log T)^{C/2}, x^{1 - (\epsilon/2)}]$. Therefore, the largest prime factor of $r$ is greater than $x^{1 - (\epsilon/2)}/(\log T)^{C/2}$. There exists some prime $p > x^{1 - (\epsilon/2)}/(\log T)^{C/2}$ such that $p^2 | P(n)$. The number of $n$ with this property is
\[\sum_{x^{2 - \epsilon}/(\log T)^C < p^2 \ll x^2} \left(\frac{x\rho(p^2)}{p^2} + O(\rho(p^2))\right).\]
We have already established that the first sum is sufficiently small. In addition,
\[\sum_{x^{2 - \epsilon}/(\log T)^C < p^2 \ll x^2} \rho(p^2) \ll \frac{x}{\log x}.\]
\end{proof}

\begin{cor} For all $\epsilon < 1.75$, the number of $n \leq x$ for which $\Omega_T (P(n)) > (1 + \epsilon) \log \log T$ is
\[O\left(\frac{x}{(\log T)^{(1 + (\epsilon/2))\log(1 + (\epsilon/2)) - (\epsilon/2)}}\right).\]
\end{cor}

\begin{proof} Let $n \leq x$. If $\Omega_T (P(n)) > (1 + \epsilon) \log \log T$, then there are two possibilities:
\begin{enumerate}
\item{$\omega_T (P(n)) > (1 + (\epsilon/2)) \log \log T$,}
\item{$\Omega_T (P(n)) - \omega_T (P(n)) > (\epsilon/2) \log \log T$.}
\end{enumerate}
By Lemma $5.2$, the number of $n$ satisfying the first condition is
\[O\left(\frac{x}{(\log T)^{(1 + (\epsilon/2)) \log(1 + (\epsilon/2)) - (\epsilon/2)}}\right).\]
Suppose $\Omega_T (P(n)) - \omega_T (P(n)) > (\epsilon/2)\log \log T$. Then, $P(n)$ has a square factor greater than $2^{(\epsilon/2)\log \log T} = (\log T)^{\epsilon(\log 2)/2}$. By the previous theorem, the number of $n$ satisfying the second condition is
\[O\left(\frac{x}{(\log T)^{\epsilon(\log 2)/4}}\right).\]
For all $\epsilon < 1.75$,
\[(1 + (\epsilon/2))\log(1 + (\epsilon/2)) - (\epsilon/2) < \epsilon(\log 2)/4.\]
Therefore, the number of $n \leq x$ for which $\Omega_T (P(n)) > (1 + \epsilon)\log \log T$ is
\[O\left(\frac{x}{(\log T)^{(1 + (\epsilon/2))\log(1 + (\epsilon/2)) - (\epsilon/2)}}\right).\]
\end{proof}

For the rest of the paper, we will let $\epsilon < 1.75$. Suppose there exist $p, q \in (T, 4ax)$ with $\Omega_T (p - 1), \Omega_T (q - 1) > A \log \log T$ and $(p - 1)(q - 1) | P(n)$. Then $\Omega_T (P(n)) > 2A \log \log T > (1 + \epsilon) \log \log T$ for $\epsilon < 2A - 1$, which we have handled with the previous theorem.

The other possibility is that $m = pr$, where $r$ is $T$-smooth and $\Omega_T (\varphi(r)) < A \log \log T$. If $r$ is $T$-smooth, then $v = \varphi(r)$ is $T$-smooth as well. Therefore, $P(n) = (p - 1)v$ with $v$ $T$-smooth. Hence,
\[P(n) = (p - 1)v < 4axT^{A \log \log T}.\]
If $T^{A \log \log T} \ll x^{1 - \delta}$ for some $\delta > 0$, then $P(n) = O(x^{2 - \delta})$, which would imply that $n = O(x^{1 - (\delta/2)})$. We find a value of $T$ for which $T^{A \log \log T}$ is very close to $x^{1 - \delta}$. We have
\[A \log T \log \log T = (1 - \delta) \log x.\]
An approximate solution is
\[T = \exp\left(\frac{1 - \delta}{A}\left(\frac{\log x}{\log \log x}\right)\right).\]
For such $T$ (for all $\delta > 0$),
\[V_3 = O\left(\frac{x}{(\log T)^{(1 + (\epsilon/2)) \log(1 + (\epsilon/2)) - (\epsilon/2)}}\right) = O\left(\frac{x}{(\log T)^{(A + (1/2))\log(A + (1/2)) - A + (1/2) - \delta}}\right).\]
Note that $V_1$ is independent of $T$, whereas $V_1$ and $V_2$ decrease as $T$ increases. In order to let $T$ be as large as possible, we use the formula for $T$ above for the rest of the paper.

\section{The number $p$ is small}

Suppose that if $\varphi(m) = P(n)$, then $m$ is $T$-smooth. We use an argument similar to the one at the end of the previous section to show that the number of such $n$ is negligible. By Theorem $5.3$, we may assume that $\Omega_T (P(n)) < A \log \log T$. In addition, $P(n)$ is $T$-smooth because $m$ is $T$-smooth. Hence,
\[P(n) < T^{A \log \log T} = o(x).\]
So, we may assume that $n = o(x^{1/2})$. We may ignore such $n$.

\section{Optimizing parameters}

Here are the bounds we obtained (for all $\delta > 0$):
\[V_1 = O\left(\frac{x}{(\log x)^{1 - e(\log 2)/2 - \delta}}\right),\]
\[V_2 = O\left(\frac{x}{(\log T)^{A \log A - A + 1}}\right),\]
\[V_3 = O\left(\frac{x}{(\log T)^{(A + (1/2))\log(A + (1/2)) - A + (1/2) - \delta}}\right).\]
The previous section states that if $\varphi(m) = P(n)$, then we may assume that $m$ is not $T$-smooth. Therefore, $V_P (x)$ is at most the sum of our upper bounds for $V_1$, $V_2$, and $V_3$.

We now optimize our bounds for $V_2$ and $V_3$. As $A$ increases, $V_2$ increases and $V_3$ decreases. We set $V_2$ and $V_3$ approximately equal:
\[\frac{x}{(\log T)^{A \log A - A + 1}} = \frac{x}{(\log T)^{(A + (1/2))\log(A + (1/2)) - A + (1/2)}},\]
which implies that
\[A \log A - A + 1 = (A + (1/2))\log(A + (1/2)) - A + (1/2) .\]
The solution is $A \approx 0.76$. Plugging in this value shows that
\[V_2 + V_3 \ll \frac{x}{(\log T)^{0.0312 - \delta}}.\]
Recall that $T = \exp(((1 - \delta)/A)(\log x/\log \log x))$. Therefore,
\[V_P (x) = O\left(\frac{x}{(\log x)^{0.0312 - \delta}}\right).\]

\section{Conclusion}

One short proof that the range of Euler's $\varphi$-function has density zero uses the following result of Erd{\H o}s and Wagstaff \cite{EW}.
\begin{thm} For all $\epsilon > 0$, there exists a $T = T(\epsilon)$ such that the upper density of numbers $n$ for which $p - 1 | n$ for some $p > T$ is less than $\epsilon$.
\end{thm}

\begin{cor} We have $V(x) = o(x)$.
\end{cor}

\begin{proof} Let $\epsilon > 0$. Suppose $\varphi(m) = n$. There exists a $T$ such that
\[\overline{\mathbf{d}} (\{n : \exists p > T\ \textrm{s.t.}\ p - 1 | n\}) < \epsilon/2.\]
Suppose $n$ has no such divisor $p - 1$. Then, $m$ is $T$-smooth. Therefore, $n$ is $T$-smooth as well. The density of $T$-smooth numbers is $0$. For any $\epsilon > 0$, the upper density of $\varphi(\Z_+) < \epsilon$. Hence, $V(x) = o(x)$.
\end{proof}

If we wanted to do a similar argument for polynomials in the range of the $\varphi$-function, we would need to prove the following variant of Theorem $8.1$.

\begin{conj} For all $\epsilon > 0$ and all polynomials $P(x)$, there exists a $T = T(\epsilon, P)$ such that the upper density of numbers $n$ for which $p - 1 | P(n)$ for some $p > T$ is less than $\epsilon$.
\end{conj}

Though we have already showed that $V_P (x) = o(x)$ for irreducible quadratic $P$, the conjecture still possesses independent interest in this case. We also ask what bounds one can obtain when $P$ is reducible.

\end{document}